\documentclass{amsart}
\usepackage{amsmath,amsfonts,amsthm}
\usepackage{mathrsfs}
\usepackage{amssymb}
\usepackage{bm}
\usepackage{cite}
\usepackage{graphicx, epsfig}
\usepackage{hyperref}
\usepackage{soul}
\usepackage{xcolor}

\usepackage[margin=3cm]{geometry}

\newtheorem{theorem}{Theorem}[section]
\theoremstyle{plain}

\newtheorem{corollary}[theorem]{Corollary}

\newtheorem{lemma}[theorem]{Lemma}

\newtheorem{remark}[theorem]{Remark}

\numberwithin{equation}{section}

\usepackage{hyperref}
\hypersetup{
    colorlinks   = true,
    urlcolor     = blue,
    citecolor    = red,
    bookmarksopen=true
}

\title{Estimates on elliptic equations that hold only where the Hessian is large}
\author{Amit Kumar Acharya}
\address[Amit Kumar Acharya]{SRM University Amaravati, Andhra Pradesh-522502, India}
\email{amitacharya.h@gmail.com, amitkumar\_acharya@srmap.edu.in}
\author{Ram Baran Verma}
\address[Ram Baran Verma]{SRM University Amaravati, Andhra Pradesh-522502, India}
\email{rambv88@gmail.com, rambaran.v@srmap.edu.in}

\begin{document}
\begin{abstract}
In this article, we establish H\"older regularity for viscosity solutions to a class of degenerate fully nonlinear elliptic equations of the form
\[
F(D^2u,Du)=f(x)~~\text{in}~~B_1,
\]
where the operator is elliptic only in regions where the Hessian is sufficiently large. Such equations arise naturally in free boundary problems and models with partial ellipticity. The proof combines a modified cusp function with a decomposition of the contact set in a point-to-measure argument. As a consequence, interior H\"older continuity follows under natural structural assumptions.
\end{abstract}
\maketitle
\section{Introduction}
In this article, we study regularity properties of viscosity solutions to a class of degenerate fully nonlinear elliptic equations of the form
\begin{equation}\label{aleso2}
F(D^{2}u,Du)=f \quad \text{in}~~ B_{1},
\end{equation}
where the operator $F$ is elliptic only in regions where the Hessian is sufficiently large. More precisely, ellipticity is assumed to hold only when $|D^{2}u|>\kappa$, while no uniform ellipticity is imposed in the complementary region. Such equations arise naturally in problems where the governing equation is active only on part of the domain, for instance in free boundary problems and constrained variational models.\\
A motivating example is provided by Figalli and Shahgholian \cite{figalli2014general}, where the authors consider solutions satisfying
\begin{equation*}\label{ales1}
\left\{
\begin{aligned}
|F(D^{2}u)|&\leq C_{0} \quad \text{a.e. in } \Omega \cap B_{1},\\
|D^{2}u|&\leq \kappa_{0} \quad \text{a.e. in } B_{1}\setminus \Omega,
\end{aligned}
\right.
\end{equation*}
with $\Omega \subset B_{1}$ open. In this formulation, the equation is enforced only in the region where the Hessian is large, while elsewhere the solution satisfies a constraint. This naturally leads, in the viscosity sense, to inequalities of the form
\begin{equation}\label{al3}
\left\{
\begin{aligned}
\mathcal{L}^{-}(D^{2}u,Du)&\leq C_{0},\\
\mathcal{L}^{+}(D^{2}u,Du)&\geq -C_{0},
\end{aligned}
\right.
\end{equation}
where ellipticity is present only outside a fixed neighborhood of the origin in the space of symmetric matrices, where $\mathcal{L}^{\pm}$ is defined in \eqref{hoill}. The regularity theory for uniformly elliptic fully nonlinear equations is classical and well developed see, for instance, Caffarelli and Cabré \cite{caffarelli1995fully}. However, when ellipticity holds only in a restricted region, these methods cannot be directly applied. In particular, the Krylov--Safonov theory and the Aleksandrov--Bakelman--Pucci estimate rely fundamentally on uniform ellipticity, see\cite{krylov1981certain,krylov1983boundedly}.
For degenerate elliptic equations, analogous estimates have been established in \cite{imbert2016estimates,savin2007small}. In \cite{savin2007small}, Savin proved H\"{o}lder regularity for a class of degenerate elliptic equations that are uniformly elliptic whenever the solution, its gradient, and its Hessian are sufficiently small. The proof is based on the method of sliding paraboloids. On the other hand, Imbert and Silvestre \cite{imbert2016estimates}, established H\"{o}lder regularity for equations that are elliptic only when the gradient is sufficiently large. Their proof is based on a refined point-to-measure estimate obtained by sliding a singular cusp function
\[
\eta(x)=-|x|^{1/2}
\]
from below. At every contact point, the gradient of the solution coincides
with the gradient of the test function. Since the equation is assumed to be
uniformly elliptic whenever the gradient is large and $Du=D\eta=-\frac{1}{2}|x|^{-3/2}x$ at all the contact point away from the origin. As $|D\eta|$ is large in a suitable deleted neighbourhood of the origin so we can use the equation to get uniform estimate on the Jacobian. So we get the result.\\
The present work deals with a fundamentally different situation, where
ellipticity depends on the size of the Hessian rather than the gradient.
Although at the touching point(away from the origin) gradient of the solution and test function matches but Hessian does not matches. Precisely,
\begin{equation}\label{nega11}
\left\{
\begin{aligned}
&Du(z)=D\eta(z)\\
&\text{and}\\
&D^{2}u(z)\geq D^{2}\eta(z).
\end{aligned}
\right.
\end{equation}
As our considered equation is uniformly elliptic when Hessian is large, so for simplifying the calculation we can use $\eta(x)=-10|x|^{3/2}.$ But despite this choice, by using the equation \eqref{nega11}, we only get $[D^{2}u(z)]^{-}\leq D^{2}\eta(z).$ This bound does not give any information on the size of Hessian like in the case when dealing with large gradient case as in \cite{imbert2016estimates}. The present result extends the point-to-measure strategy of Imbert and Silvestre \cite{imbert2016estimates} to a fundamentally different
type of degeneracy. Unlike the gradient-dependent setting, the equation cannot be invoked directly at contact points because the touching condition provides only a one-sided comparison of the Hessian. Our approach overcomes this difficulty by combining geometric information with a decomposition of the
contact set into regions where the equation is applicable and where purely geometric arguments suffice. This yields a new mechanism for obtaining uniform
Jacobian estimates in Hessian-dependent degenerate elliptic equations. The main novelty of the paper is a new way of overcoming this difficulty. Instead of attempting to use the equation everywhere, we decompose the contact set into two disjoint parts according to the size of the Hessian of the
solution, that is, $A=G\cup B$ where
\begin{equation*}
\left\{
\begin{aligned}
&G=\{z\in~A:\ |D^{2}u(z)|>\kappa\},\\
&B=\{z\in A:\ |D^{2}u(z)|\le\kappa\},
\end{aligned}
\right.
\end{equation*}
$A$ is the set of contact point and $\kappa$ is a parameter in the definition of operator, see \eqref{hoill}. The applicability of this decomposition method relies on the interesting observation that for a suitable choice of $\kappa$ we can get uniform bound on the Jacobian of transport map (which maps the contact set to the set of vertices) on the bad part of the contact set without even using the equation, see \eqref{jac111}. While on the good part $G$ of the contact set we can use the equation to get bound on the Jacobian of the transport map. Inspired from the above discussion and \cite{imbert2016estimates}, this article establishes the following theorem
\begin{theorem}\label{mainthm}
Let $u\in C(\overline{B}_{1})$ satisfy \eqref{al3} and $\|u\|_{L^{\infty}(B_{1})}\leq C_{0}$. Then $u\in C^{\alpha}(B_{1/2})$ and
\[
\|u\|_{C^{\alpha}(B_{1/2})}\leq C C_{0},
\]
where $\alpha$ depends on $\lambda,\Lambda,n$ and $C$ depends on $\kappa/C_{0},\lambda,\Lambda,n$.
\end{theorem}
The rest of the paper is organized as follows. In Section 2 we collect the necessary definitions and preliminary results. Section 3 is devoted to the proof of the H\"{o}lder estimate via a point-to-measure argument. 
\section{Auxiliary definitions and  results}
Given three positive constants $\lambda,\Lambda$ and $\kappa$ we define the following degenerate elliptic operator
\begin{equation}\label{hoill}
\mathcal{L}_{(\lambda, \Lambda),\kappa}^{\pm}\big(X,p\big)=\left\{
\begin{aligned}
&\Lambda\text{Tra}(X)^{\pm}-\lambda\text{Tra}(X)^{\mp}+\Lambda|p|~~\quad \text{if}~~|(X)|>\kappa,\\ 
&\pm\infty~~\quad \text{if}~~|(X)|\leq \kappa,\\
\end{aligned}
\right.
\end{equation}
for real symmetric matrix $X$ and $p\in \mathbb{R}^{n}.$ From here on wards $\lambda$ and $\Lambda$ will be fixed and constants depending on $n,\lambda,\Lambda$ are called universal constants. Therefore, we denote $\mathcal{L}_{(\lambda, \Lambda),\kappa}^{\pm}\big(D^{2}u,Du\big)$ by $\mathcal{L}_{\kappa}^{\pm}\big(D^{2}u,Du\big).$\\
\textbf{Scaling:}~The operator $\mathcal{L}_{\kappa}^{-}\big(D^{2}u,Du\big)$ satisfies the following scaling property:
\begin{enumerate}
\item{} If $u\in C^{2}(\Omega)$ $\mathcal{L}_{\kappa}^{-}\big(D^{2}u,Du\big)\leq C_{0}$~in~$\Omega$
then for $r,\gamma\in[0,1],$ $u_{r}(x)=\gamma u(x_{0}+rx)$  satisfies $\mathcal{L}_{r^{2}\gamma\kappa}^{-}\big(D^{2}u_{r},Du_{r}\big)\leq \gamma C_{0} r^{2}$~in~$x_{0}+r\Omega,$
where
\begin{equation*}
\mathcal{L}_{r^{2}\gamma\kappa}^{-}\big(X,p\big)=\left\{
\begin{aligned}
&\lambda\text{Tra}(X)^{+}-\Lambda\text{Tra}(X)^{-}-r\Lambda|p|~~\quad \text{if}~~|(X)^{-}|>r^{2}\kappa\gamma,\\ 
&-\infty~~\quad \text{if}~~|(X)^{-}|\leq r^{2}\kappa\gamma r.\\
\end{aligned}
\right.
\end{equation*}
\item{} We also have the following relation $\kappa_{1}\geq \kappa_{2}$
\begin{equation*}\label{order}
\mathcal{L}_{\kappa_{1}}^{+}\big(X,p\big)\geq \mathcal{L}_{\kappa_{2}}^{+}\big(X,p\big)
\geq \mathcal{L}_{\kappa_{2}}^{-}\big(X,p\big)\geq \mathcal{L}_{\kappa_{1}}^{-}\big(X,p\big),\\
\end{equation*}
for all $p\in \mathbb{R}^{n}$ and $n\times n$ real symmetric matrix $X.$
\end{enumerate}
\begin{remark}\label{remsc}
As a consequence of above relation we find that if $\mathcal{L}_{\kappa}^{-}\big(D^{2}u,Du\big)\leq C_{0}$~in~$\Omega$ then 
$\mathcal{L}_{\kappa}^{-}\big(D^{2}u_{r},Du_{r}\big)\leq \gamma C_{0}r^{2}$~in~$x_{0}+r\Omega$ \text{for} $r,\gamma\in[0,1].$ 
\end{remark}
As we know that H\"{o}lder estimate is a consequence of Harnack's inequality. But simple observation shows that we do not need full Harnack inequality for the proof of H\"{o}lder estimate. In fact, we can derive it from $L^{\epsilon}$ estimate which is going to be strategy here. Because the operator considered here degenerate in nature. In the proof of the interior $L^{\epsilon}$ estimate we need the following ink spot lemma. For its proof see Lemma 2.1\cite{imbert2016estimates}.
\begin{lemma}\label{covering1}
Let $G\subset H\subset B_{1}$ be two open sets. Suppose also that for some $\delta\in(0,1)$ following two assumptions hold:
\begin{enumerate}
\item{} If for any ball $B\subset B_{1}$ satisfying $|B\cap G|>(1-\delta)|B|,$ then $B\subset H.$
\item{}$|G|\leq (1-\delta)|B_{1}|.$
\end{enumerate}
Then $|G|\leq (1-c\delta)|H|$ for some constant depending only on the dimension.
\end{lemma}
\section{Proof of main result}
\subsection{Point to Measure Estimate}
\begin{theorem}\label{measure}
There exist two small constants $k_{0}>0,$ $\delta>0$ and a large constant $M>0,$ such that if $k\leq k_{0},$ then for any lower semicontinuous function $u:B_{1}\longrightarrow\mathbb{R}$ satisfying 
\begin{equation}\label{meac}
\left\{
\begin{aligned}
&(i)&~u\geq0 \quad \text{in}~B_1,\\
&(ii)&~\mathcal{L}_{\kappa}^{-}\big(D^{2}u,Du\big)\leq 1\quad \text{in}~B_1,\\ 
&(iii)&~~|\{u>M\}\cap B_{1}|>(1-\delta)|B_{1}|,\\
\end{aligned}
\right.
\end{equation}
then $u>1$ in $B_{1/4}.$
\end{theorem}
\begin{proof} Following the ideas of \cite{imbert2016estimates}, we divide the proof in three parts. First we establish the result for classical solution, then for semiconcave and finally lower semicontinuous supersolution.\\
\textbf{Proof when $u$ is $C^{2}$:} Before we proceed for the poof of the lemma, we would like to emphasize that we can use our equation only when hessian is large. Suppose by contradiction for each sufficiently small $\kappa,$ $\delta$ and large $M,$ we can find a point $x_{0}\in B_{\frac{1}{4}}$ such that $u(x_{0})\leq 1.$ Let us set $V=\{x\in B_{\frac{1}{4}}~:~u(x)>M\},$ where $M$ will be chose below. Fix $\bar x\in V$ and define
\[\psi(y)=u(y)-\eta(y-\bar x),\]
where $\eta(y)=-10|y|^{3/2}.$ We look for the point of minimum for the function $\psi.$ As $u\ge0$ and
\[u(x_{0})-\eta(x_{0}-\bar x)\leq 1+10(1/2)^{3/2},\]
while
\[u(y)-\eta(y-\bar x)\geq10(3/4)^{3/2}~~\text{for}~~y\in\partial B_{1}.\]
Minimum of $\psi$ cannot occur on $\partial B_{1}.$ Therefore there exists $z\in B_{1}$ such that
\begin{equation}\label{reffi}
u(z)-\eta(z-\bar x)=\min_{B_{1}}\big(u-\eta(\cdot-\bar x)\big).
\end{equation}
Define the contact set
\[A=\left\{z\in B_{1}:\exists~~\bar x\in V~\text{satisfying}~\eqref{reffi}~\right\}.\]
By minimality,
\[u(z)\le u(x_{0})-\eta(x_{0}-\bar x)\leq1+\frac{5}{\sqrt{2}}.\]
 So if we choose $M>1+\frac{5}{\sqrt{2}},$ then $\mathcal A\subset\{u\leq M\}$ and $\bar{x}\not=z.$ Thus the function $\eta$ is smooth.\\
Therefore at the point of minimum we have
\begin{equation}\label{onel}
 Du(z)-D\eta(z-\bar x)=0.
 \end{equation}
Define
\[T:A\to V,\qquad T(z)=\bar x.\]
With this notation \eqref{onel} can be rewritten as $Du(z)-D\eta(z-T(z))=0.$ Differentiating this equation again we have
\[D^{2}u(z)-D^{2}\eta(z-T(z))(I-DT(z))=0.\]
Rewriting this equation we have
\begin{equation}\label{jaco}
DT(z)=I+(D^{2}\eta(z-T(z)))^{-1}D^{2}u(z).
\end{equation}
In order to get contradiction, we use the Area formula, for this we estimate the Jacobian of $DT(z).$ This will be achieved by using the equation. But in our case we can not use the equation as in \cite{imbert2016estimates} because at the point of minimum we do not have $D^{2}u(z)=D^{2}\eta(z-\bar{x}),$ instead we have $D^{2}u(z)\geq D^{2}\eta(z-\bar{x}).$ This can only help us to get $[D^{2}u(z)]^{-}\leq D^{2}\eta(z-\bar{x})$ which does not give estimate for $|D^{2}u(z)|$ so we can not use the equation directly.\\ 
In order to overcome this issue we divide our contact set into two parts.\\
First part is $G=\{z\in A:|D^{2}u(z)|>\kappa\}$ which is a good set as we can use the equation to estimate the Jacobian of $T$. Other part is $B=\{z\in A:|D^{2}u(z)|\le\kappa\}$
is bad from the point of view that we can not use the equation here. But interesting part is that Jacobian is automatically bounded on this part.\\
\textbf{Case 1: $z\in G$}\\
Since $u+\eta(\cdot-\bar x)$ has a minimum at $z,$ $D^{2}u(z)+D^{2}\eta(z-\bar x)\ge0.$
Consequently,
\[\operatorname{Tr}(D^{2}u(z))^{-}\leq\operatorname{Tr}(D^{2}\eta(z-\bar x)).\]
Note also that $|D^{2}(z)|>\kappa$ therefore, we can use \eqref{meac}(ii), thus we have
\[\lambda\operatorname{Tr}(D^{2}u)^{+}-\Lambda\operatorname{Tr}(D^{2}u)^{-}-\Lambda|Du|\leq1.\]
Using $Du=-D\eta,$ we obtain
\[\operatorname{Tr}(D^{2}u)^{+}\leq C\Big(1+\operatorname{Tr}(D^{2}\eta)+|D\eta|\Big).\]
Observe that
\begin{equation*}
\left\{
\begin{aligned}
&|D\eta|\leq C|z-\bar x|^{1/2}\\
&|D^{2}\eta|\leq C|z-\bar x|^{-1/2}.
\end{aligned}
\right.
\end{equation*}
Using these two expression in the above estimate we get bound of the positive part $D^{2}u(z).$ As we already have control the negative prat of hessian of $u$ in terms of Hessian of $\eta$ as $z$ is a point of minimum. Thus be have 
\[|D^{2}u|\leq C\left(1+|z-\bar x|^{-1/2}\right).\]
Now using $|(D^{2}\eta)^{-1}|\leq C|z-\bar x|^{1/2}$ and above expression in \eqref{jaco}, we get
\[|DT(z)|\leq C.\]
\textbf{Case 2: $z\in B:$}~In this case we have, $|D^{2}u(z)|\le\kappa.$ Consequently, \eqref{jaco} implies 
\begin{equation}\label{jac111}
|DT-I|\le|(D^{2}\eta)^{-1}||D^{2}u|\leq C\kappa. 
\end{equation}
Thus, for any fix $0<k_{0}<\frac1{2C}$  we have
\[|DT-I|\leq\frac{1}{2},\]
as long as $0\le\kappa\le k_{0}.$ Thus we find $|\det DT(z)|\leq\left(\frac32\right)^n$ in this case also. Applying the area formula,
\[|V|\le\int_{A}|\det DT(z)|dz\leq C|A|.\]
Consequently, $|B_{1/4}|-\delta |B_{1}|\le C\delta |B_{1}|.$ Choosing
\[\delta<\frac{|B_{1/4}|}{(C+1)|B_{1}|}\]
gives a contradiction. Therefore $u>1$ in~~$B_{1/4}.$\\
\textbf{Proof of Theorem \eqref{measure} when $u$ is semiconcave function and satisfying \eqref{meac}:}\\
The proof of theorem in this part is further divided into many parts.
\begin{enumerate}
\item~~\textbf{Semiconcavity and second-order differentiability.}\\
In this case we assume that for any choice of $\kappa_{0}, M,\delta,$ there exists a $u:B_{1}\rightarrow\mathbb{R}$ semiconcave function satisfying \eqref{meac} and $\min_{B_{1/4}}u\leq 1.$ As we have assumed that $u$ is semiconcave, so there exists a constant $C_{0}>0$ such that
\[x\longmapsto u(x)-\frac{C_{0}}{2}|x|^{2}\]
is concave in $B_{1}.$ Consequently, for every $x\in B_{1}$ there exists a vector $p\in D^{+}u(x)$ (the superdifferential of $u$ at $x$) satisfying
\begin{equation}\label{SC1}
u(y)\leq u(x)+p\cdot (y-x)+\frac{C_{0}}{2}|y-x|^{2}~~\forall~~y\in B_{1}.
\end{equation}
Furthermore, by Alexandroff theorem the function is twice differentiable almost everywhere. By this we mean, there exists a set $E\subset B_{1}$ with $|E|=0$ and for every point of $x\in B_{1}\setminus E$  the function $u$ is twice differentiable in the classical sense at $x$. More precisely, for every $x\in B_{1}\setminus E$ there exists a symmetric matrix $D^{2}u(x)$ satisfying
\[u(y)=u(x)+Du(x)\cdot(y-x)+\frac12\langle D^{2}u(x)(y-x),\,y-x\rangle+o(|y-x|^{2})\]
as $y\to x$. Furthermore by following \cite{imbert2016estimates}, we know that the semiconcavity implies the following first-order expansion
\begin{equation}\label{gradexp}
Du(y)=Du(x)+D^{2}u(x)(y-x)+o(|y-x|).
\end{equation}
In the above expression whenever $u$ is not differentiable at $y,$ $Du(y)$ denotes any element of the superdifferential $D^{+}u(y).$ Identity \eqref{gradexp} shows that the gradient is differentiable
almost everywhere with derivative $D^{2}u$. Therefore every computation involving $Du$, $D^{2}u$ and the transport map will be carried out on the full measure set $B_{1}\setminus E.$
\\
\item{} \textbf{Construction of the contact set:} As we have started by assuming that there exists a point $x_{0}\in B_{\frac{1}{4}}$ such that $u(x_{0})\leq1.$ Let us set
\[V:=\{x\in B_{\frac14}:u(x)>M\},\]
where the constant $M>1$ will be chosen later. Form the given condition we have
\[|V|\geq|B_{\frac14}|-\delta|B_{1}|,\]
consequently, if we choose $\delta$ sufficiently small then $V$ has positive measure. For each point in $\bar x\in V,$
 define
\[\Psi_{\bar x}(y):=u(y)-\eta(y-\bar x)~~\text{for}~~y\in B_{1},\]
where $\eta(x)=-10|x|^{\frac32}.$ Notice that $\Psi_{\bar{x}}$ is a lower semicontinuous function therefore, there exists a point
$z=z(\bar x)\in\overline{B_{1}}$ such that
\begin{equation*}\label{eq:contact-minimum}
\Psi_{\bar x}(z)=\min_{y\in B_{1}}\Psi_{\bar x}(y).
\end{equation*}
We claim that $z$ lies in the interior of $B_{1}.$ In fact, for $z\in B_{1}$ we have 
\[
\begin{aligned}
u(z)-\eta(z-\bar x)&\leq u(x_{0})-\eta(x_{0}-\bar x)\\
&\leq1+10|x_{0}-\bar x|^{3/2}=1+\frac{5}{\sqrt{2}},
\end{aligned}
\]
 where we have used  $|x_{0}-\bar x|\leq 1/2.$  On the other hand for $z\in\partial B_{1}$ we have 
\[
\begin{aligned}
u(z)-\eta(z-\bar x)&\ge10\left(\frac34\right)^{3/2}\\
&=\frac{15\sqrt3}{4},
\end{aligned}
\]
where we have used $|z-\bar x|\geq3/4$ and $u\geq 0.$ As $1+\frac5{\sqrt2}<
\frac{15\sqrt3}{4},$ so $z\in B_{1}.$ Moreover by choosing $M>
1+\frac5{\sqrt2},$ we find that $z\not=\bar{x}.$  Consequently,
the cusp function
\[y\longmapsto\eta(y-\bar x)\]
is of class $C^{2}$ in a neighbourhood of the contact point.
\[A:=\left\{z\in B_{1}:\exists~\bar x\in V
~\text{such that}~u(z)-\eta(z-\bar x)=\min_{y\in B_{1}}
\left(u(y)-\eta(y-\bar x)\right)\right\}.\]
As $u(z)<M,$ so we have
\[|A|\le|\{u\le M\}\cap B_{1}|\le\delta|B_{1}|.\]
\item{}\textbf{Differentiability at the contact points and the transport map:}
Let $z\in A$ be an arbitrary contact point and $\bar x\in V$ be the corresponding point such that
\begin{equation}\label{contact-point}
u(z)-\eta(z-\bar x)=\min_{y\in B_{1}}\left\{u(y)-\eta(y-\bar x)\right\},
\end{equation}
In view of \eqref{contact-point} and $z\not=\bar{x}$, it is easy to see that $\phi(y):=u(z)+\eta(z-\bar x)-\eta(y-\bar x)$ is $C^{2}$ function touching $u$ from below at $z.$ As, $u$ is a semiconcave function, therefore $u$ is differentiable at $z.$ Thus we find that $\Psi(y)=u(y)-\eta(y-\bar x)$ is a differentiable function and attains its minimum at $z,$ so we have
\begin{equation}\label{tref}
Du(z)-D\eta(z-\bar x)=0.
\end{equation}
Substituting the expression of $D\eta$ we get  $Du(z)=-15|z-\bar x|^{-1/2}(z-\bar x).$ Notice also that $x\to x\sqrt{|x|}$ is one to one map on $\mathbb{R}^{n}\setminus\{0\}$ consequently, touching vertex associated with a fixed contact point is unique. Thus we can define
\[T:A\longrightarrow V~~~~\text{by}~~~~T(z)=\bar x.\]
With this notation, \eqref{tref} can be rewritten as $Du(z)=D\eta\bigl(z-T(z)\bigr)~~\text{for}~~z\in A.$
In order to calculate the second derivative of $\eta$ we need  $z\not=T(z).$ In fact, it is easy to see that there exits a universal constant $\varepsilon_{0}>0$ depending on modulus of continuity and $M$ such that
\[|z-T(z)|\ge\varepsilon_{0}.\]
Indeed, suppose by contradiction that for each positive $\epsilon,$  $|z-T(z)|\leq\varepsilon.$ Then we have 
\[\omega(|z-T(z)|)\geq u(T(z))-u(z)\geq M-\left(1+\frac{5}{\sqrt2}\right)=\mu,\]
where $\omega$ is a modulus of continuity of $u.$ So if we choose $\epsilon_{0}$ sufficiently small such that $\omega(\epsilon_{0})<\mu/2$ then we get a contradiction to the fact that $\mu\geq\mu/2$ provided $|T(z)-z|\leq\epsilon_{0}.$ Thus the claim follows.\\
In view of above observation and the choice of $\eta$ we have 
\[\frac{15}{2\sqrt2}I\leq D^{2}\eta(z-T(z))\leq15\varepsilon_{0}^{-1/2}I.\]
Thus every eigenvalues of $D^{2}\eta(z-T(z))$ bounded below and we have
\begin{equation*}\label{inverse-bound}
\left|\left(D^{2}\eta(z-T(z))\right)^{-1}\right|\leq\frac{1}{c}.
\end{equation*}
\item{}\textbf{Lipschitz continuity of $Du$ on $A$:}~ As remarked above that there is a $\epsilon_{0}>0$ such that 
$|x-y|>\epsilon_0$ for any pair of vertex and contact point. Therefore, we find that $|D^{2}\eta(z)|\leq C|z-\bar{x}|^{-\frac{1}{2}}\leq C\epsilon_{0}^{-1/2}$ for $|x-y|>\epsilon_{0}.$ Let $(x_{1},z_{1})$ and $(x_{2},z_{2})$ be the pair of points satisfying \eqref{contact-point}. Set $s=2|z_{1}-z_{2}|$ and observe that for any $z\in B_{s}(z_{1})$ we have 
\begin{equation}\label{fil1}
\begin{aligned}
u(z)&\geq \eta(z-x_{1})\geq \xi(z_{1}-x_{1})+\langle D\eta(z_{1}-x_{1}),(z-z_{1})\rangle-Cs^{2}\epsilon_{0}^{-\frac{1}{2}}\\
=& u(z_{1})+\langle Du(z_{1}),(z-z_{1})\rangle-Cs^{2}\epsilon_{0}^{-\frac{1}{2}},
\end{aligned}
\end{equation}
where we have used the bound of $D^{2}\eta$ and \eqref{contact-point}. Consequently, for $z=z_{2}$ we have 
\begin{equation*}
u(z_{2})\geq u(z_{1})+\langle Du(z_{1}),(z_{2}-z_{1}) \rangle-Cs^{2}\epsilon_{0}^{-\frac{1}{2}}.
\end{equation*}
Interchanging the role of $z_{1}$ and $z_{2}$ we find
\begin{equation}\label{fil2}
u(z_{1})\geq u(z_{2})+\langle Du(y_{2}),(y_{1}-y_{2}) \rangle-Cs^{2}\epsilon_{0}^{-\frac{1}{2}}.
\end{equation}
From \eqref{fil1} and \eqref{fil2} we find 
\[u(z)\geq u(z_{2})+Du(z_{2}).(z_{1}-z_{2})+\langle Du(z_{1}),(z-y_{1})\rangle-C\epsilon_{0}^{-\frac{1}{2}}s^{2}.\]
On the other hand from \eqref{SC1}, we also have 
\[u(z)\leq u(z_{2})+\langle Du(z_{2}),(z-z_{2})\rangle+Cs^{2}.\]
Now by subtracting above two inequalities 
\[\langle Du(z_{1}-Du(z_{2}))(z-z_{1})\rangle\leq C(1+\epsilon_{0}^{-\frac{1}{2}})s^{2}.\]
Due to arbitrariness of $z\in B_{s}(z_{1})$ we conclude that 
\[|Du(z_{1})-Du(z_{2})|\leq C(1+\epsilon_{0}^{-\frac{1}{2}})s.\]
Thus we find that $Du$ is Lipschitz continuous but $[Du]_{Lip}$ is not universal as it depends on $\epsilon_{0}.$\\
\textbf{The map $T:A\rightarrow V:$}~As we mentioned above $u$ is differentiable at the contact point because $u$ is semiconcave. Moreover, the contact point and vertex are not same so $\eta$ is also differentiable at the contact point and 
\[Du(z)=D\eta(z-\bar{x})=-15(z-\bar{x})|z-\bar{x}|^{-\frac{1}{2}},\]
uniquely determines the value of $z-\bar{x}.$ In particular, for each $z\in V$ there is a unique point $x\in A.$ This defines a map $T:A\rightarrow V$ which is implicitly given by  $Du(z)=D\eta(z-T(z)).$
This in turn implies that $DT(z)=z-(D^{2}\eta(z-T(z)))^{-1}(Du(z)),$ where $(D\eta)^{-1}$ is the inverse of the function $D\eta:\mathbb{R}^{n}~\rightarrow~\mathbb{R}^{n}.$ Following the same calculation as on the page 1330\cite{imbert2016estimates} and replacing the power of $|z-\bar{x}|$ appropriately we have 
\begin{equation*}\label{incon}
|V|=|T(A)|=\int_{A}|\det DT(z)|dz.
\end{equation*}
By following the idea as in the classical case we want to get a contradiction. But up to now we have estimate on the Jacobian which is not universal. In the next step we get a universal estimate for Jacobian map. This is the place where we use our idea to break the contact set into good and bad parts.
\item{}\textbf{Universal Estimate on $DT$:} Recall that the Jacobian $DT,$ is given by
\begin{equation}\label{chanv}
   DT(z)=z-(D^{2}\eta(z-T(z)))^{-1}(D^{2}u(z)) 
\end{equation}
We achieve this by using \eqref{meac}(ii). As mentioned in Step 0, $u$ is pointwise twice differentiable except on a set $E$ of measure zero. In particular if $z\in A,$ we have 
\begin{equation}\label{mean2}
 \mathcal{L}_{\kappa}^{-}\big(D^{2}u(z),Du(z)\big)\leq 1. 
\end{equation}
Note that at contact point we have
\begin{equation}\label{aou}
\left\{
\begin{aligned}
&Du(z)=D\eta(z-T(z)),\\
&D^{2}u(z)\geq D^{2}\eta(z-T(z)).
\end{aligned}
\right.
\end{equation}
In order to use \eqref{mean2}, we have to ensure $|D^{2}u|\leq \kappa,$ which is not possible from \eqref{aou}. At this point we we follow the same approach as in the classical case, that is, we consider two possibilities 
\[
\left\{
\begin{aligned}
&G=\{z\in A~~|~~|D^{2}u(z)|>\kappa\},\\
&B=\{z\in A~|~|D^{2}u|\leq \kappa\},
\end{aligned}
\right.
\]
where $\kappa$ is chosen below, see \eqref{chnab}.  Suppose first that $z\in G.$ As $u-\eta(\cdot-T(\cdot))$ attain its minimum at $z$ so we have 
\[D^{2}u(z)-D^{2}\eta(z-T(z))\ge0.\]
Thus 
\begin{equation}\label{negative}
 \text{Tr}(D^{2}u(z))^{-}\leq -D^{2}\eta(z-T(z)).   
\end{equation}
Using \eqref{mean2}, we have 
\[\lambda\operatorname{Tr}\bigl(D^{2}u(z)\bigr)^{+}-\Lambda\operatorname{Tr}\bigl(D^{2}u(z)\bigr)^{-}-\Lambda|Du(z)|\le1.\]
Hence,
\begin{equation}\label{positive}
\operatorname{Tr}\bigl(D^{2}u(z)\bigr)^{+}\leq C\big(1+\operatorname{Tr}\bigl(D^{2}u(z)\bigr)^{-}+|Du(z)|\big).  \end{equation}
By \eqref{negative}, \eqref{positive} and $|D^{2}\eta(z-T(z))|\leq C|z-T(z)|^{-1/2}$ we have 
\begin{equation*}
|D^{2}u(z)|\leq C\big(1+|z-T(z)|^{-\frac{1}{2}}\big).
\end{equation*}
Now using \eqref{chanv} and $|D^{2}\eta(z-T(z))^{-1}|\leq C|z-T(z)|^{\frac{1}{2}}$ we find that 
\[|DT(z)|\leq 1+C|z-T(z)|^{\frac{1}{2}}\big(1+|z-T(z)|^{-\frac{1}{2}}\big).\]
Thus we have $|DT(z)|\leq C.$\\
Next we assume that $z\in B$ so we have $|D^{2}u(z)|\leq\kappa,$ therefore we can not use \eqref{mean2}. But observe that \eqref{chanv} can be rewritten as follow:
\[DT(z)-I=[D^{2}\eta(z-T(z))]^{-1}D^{2}(z).\]
By using $|D^{2}\eta(z-T(z))^{-1}|\leq C|z-T(z)|^{1/2},$  $|D^{2}u(z)|\leq\kappa$ and $|z-T(z)|\leq \sqrt{2}$ we get
\begin{equation}\label{chnab}
|DT(z)-I|\leq C\kappa\sqrt{2}.
\end{equation}
Now, choosing $\kappa_{0}<[4C\sqrt{2}]^{-1},$ which is universal constant. So for any $0\leq\kappa\leq\kappa_{0},$ we have 
\[|DT(z)-I|\leq\frac{1}{2}.\]
Thus we find that all the eigenvalues of $DT(z)$ lies in $[1/2,3/2].$ Thus we have 
\begin{equation*}
|\text{Det}DT(z)|\leq \big(3/2\big)^{n}.  
\end{equation*}
Thus we got universal bound on the Jacobian of $T.$ Then we can find the contradiction as in the classical case.
\end{enumerate}
\textbf{Proof of Theorem \eqref{measure} when $u$ is lower semicontinuous function and satisfying \eqref{meac}:} This step follows on the same line as Proposition 3.4\cite{imbert2016estimates} so we do not write the details here.
\end{proof}
\subsection{Barrier function}
This section deal with the construction of barrier will be used in the next section. Let us consider the following $\rho(x)=|x|^{-\beta}.$ For $x\not=0,$ this function is differentiable and 
\begin{equation*}
\left\{
\begin{aligned}
&D\rho(x)=-\beta|x|^{-\beta-2}x,\\ 
&D^{2}\rho(x)=|x|^{-\beta-2}\Big[\beta(\beta+2)x\otimes x|x|^{-2}-\beta I\Big].\\
\end{aligned}
\right.
\end{equation*}
Observe that the eigenvalues of $D^{2}\rho(x)$ are $\beta(\beta+2)|x|^{-\beta-2}$ and $-\beta|x|^{-\beta-2}$ and  with multiplicity $1$ and $n-1$ respectively. Note that for any $x\in B_{2}\setminus\{0\}$ and large $\beta$  we have $|D^{2}\rho(x)|\geq\beta|x|^{-\beta-2},$ where for the matrix norm we have taken the maximum of the eigenvalue. Therefore, by choosing $\beta$ sufficiently large we can make the Hessian as large as we need. So in the following line we compute as if there is no restriction on the Hessian 
\begin{equation*}
\begin{aligned}
\mathcal{L}_{\kappa_{0}}^{-}\big(D^{2}\rho,D\rho\big)&=\lambda\beta(\beta+1)|x|^{-\beta-2}-\Lambda(n-1)\beta |x|^{-\beta-2}-\Lambda\beta|x|^{-\beta-1}\\
&=\beta|x|^{-\beta-2}\Big(\lambda(\beta+1)-\Lambda(n-1)-\Lambda|x|\Big)\\
&\geq\beta |x|^{-\beta-2}\Big(\lambda(\beta+1)-\Lambda(n-1)-\Lambda R_{0}\Big)
\geq |x|^{-\beta-2},
\end{aligned} 
\end{equation*}
for $x\in B_{R_{0}}\setminus\{0\}.$
\begin{lemma}\label{com}
There exists a small constant $\kappa_{0}>0$ depending on $\lambda,\Lambda$ and $n$ such that if $u\geq0$ is a supersolution of $\mathcal{L}_{\kappa}^{-}\big(D^{2}u,Du\big)\leq 1$ in $B_{2}$ and $u>M$ in $B_{1/4}$ for "some" large constant $M,$ then $u>1$ in $B_{1}.$ 
\end{lemma}
\begin{proof}
Let us consider the following auxiliary function
\[A(x):=M\frac{|x|^{-\beta}-2^{-\beta}}{2^{2\beta+1}}\]
Notice that the eigenvalues of $D^{2}A(x)$ outside the origin are $2^{-2\beta-1}M~|x|^{-\beta-2}\beta(\beta+1)$ and $-M2^{-2\beta-1}\beta|x|^{-\beta-2}$ with the multiplicities $1$ and $N-1$ respectively. So we can choose $M\geq1$ sufficiently large such that both $A$ and $|D^{2}A(x)|\geq k$ in $B_{1}.$ Note that this function is smooth in $B^{c}_{1/4}$ and $u=0$ on $\partial B_{2}$ and $A<M$ on $\partial B_{1/2}.$ Moreover, we have 
\begin{equation*}
\begin{aligned}
\mathcal{L}_{\kappa\Gamma r^{2}}^{-}\big(D^{2}A, DA\big)\geq& \frac{M}{2^{2\beta+1}}\mathcal{L}_{\kappa\Gamma r^{2}}^{-}(D^{2}\rho, D\rho)\\
\geq &\frac{M}{2^{3\beta+3}}\beta \geq 2~~\text{for}~M~\text{large}.
\end{aligned} 
\end{equation*}
Therefore, by comparison principle we have $u\geq A\geq 1$ in $B_{1}.$  In addition, for $\epsilon=\min_{B_{1/4}}\big(u/M-1\big)$ we have $u\geq (1+\epsilon)M>1$ in $B_{1}.$
\end{proof}
\begin{theorem}\label{corols}
There exist two small constants $k_{0}>0,$ $\delta>0$ and a large constant $M>1,$ such that if $k\leq k_{0},$ then for any continuous function $u:B_{2}\longrightarrow\mathbb{R}$ satisfying:
\begin{equation*}
\left\{
\begin{aligned}
u&\geq0~\text{in}~B_{2},\\
\mathcal{L}_{\kappa}^{-}\big(D^{2}u, Du\big)&\leq 1~\text{in}~B_{2},\\ 
|\{u>M\}\cap B_{1}|&>(1-\delta)|B_{1}|,\\
\end{aligned}
\right.
\end{equation*}
we have $u>1$ in $B_{1}.$
\end{theorem}
\begin{proof}
Let $M_{1}$ and $M_{2}$ be the constants from Lemmas \ref{measure}, \ref{com}, respectively. Then the function $w=u/M_{2}$ satisfies the assumption of Lemma \ref{measure} for $M_{2}\geq 1,$ in view of Remark \ref{remsc}. So we conclude that $w>1$ in $B_{1/4}$ i.e $u>M_{2}$ in $B_{1/4}.$ Finally, by applying Lemma \ref{com}, we get $u>1$ in $B_{1}.$ 
\end{proof}
In order to get the $L^{\epsilon}$ estimate we also need the following scaled version of the above lemma. 
\begin{corollary}\label{scaler}
There exist two small constants $k_{0}>0$ $\delta>0$ and a large constant $M>1,$ such that if $k\leq k_{0},$ then for any $r\leq1,$ $L\geq 1$ and any upper semi continuous function $u:\overline{B}_{r}\longrightarrow\mathbb{R}$ satisfying 
\begin{equation*}
\left\{
\begin{aligned}
u&\geq0~~~\text{in}~~B_{r},\\
\mathcal{L}_{\kappa}^{-}\big(D^{2}u, Du\big)&\leq L~~~\text{in}~~B_{r},\\ 
|\{u>LM\}\cap B_{r/2}|&>(1-\delta)|B_{r/2}|,\\
\end{aligned}
\right.
\end{equation*}
we have $u>L$ in $B_{r/2}.$
\end{corollary}
\begin{proof}
Consider the scaled function $u_{r}(x)=L^{-1}u(2^{-1}rx)$ then its satisfies 
\[\mathcal{L}_{\kappa r^{2}(4L)^{-1}}^{-}\big(D^{2}u_{r},Du_{r}\big)\leq \frac{r^{2}}{4}\leq 1~~\text{in}~B_{2}.\]
As mentioned in Remark \ref{order}, $u_{r}$ satisfies a stronger condition as $(4L)^{-1}r^{2}\kappa$ is smaller than $\kappa$ for $r\leq 1$ and $L\geq 1.$ Now we can apply Theorem \ref{corols} to get the result.
\end{proof}
\subsection{}\textbf{$\bm{L^{\epsilon}-}$estimate}
\begin{lemma}\label{lepaa}
There exist two small $\kappa_{0},$ $\epsilon>0$ such that if $\kappa\leq \kappa_{0},$ then for any lower semi continuous function $u: B_{2}\longrightarrow\mathbb{R}$ such that
\begin{equation}\label{lepsilon}
\left\{
\begin{aligned}
u&\geq0~~\text{in}~~B_{2},\\
\mathcal{L}_{\kappa}^{-}\big(D^{2}u,Du\big)&\leq 1~~\text{in}~~B_{2},\\ 
\inf_{B_{1}}u&\leq 1,\\
\end{aligned}
\right.
s\end{equation}
we have 
\[|\{u>t\}\cap B_{1}|\leq \tilde{C}t^{-\epsilon}~~\text{for~all}~~t>0.\]
\end{lemma}
\begin{proof}
We show 
\[|\{u>M^{k}\}\cap B_{1}|\leq C M^{-\epsilon k}~~~\text{for~all}~~k\in \mathbb{N},\]
where $M$ is the constant from Lemma \ref{scaler} and $\epsilon>0$ will be chosen below. The conclusion follows from this in a standard way.\\
Set $A_{k}:=\{u>M^{k}\}\cap B_{1},$ which are open sets in view of lower semicontinuity of $u.$ Moreover, in view of the assumption $\inf_{B_{1}}u\leq 1$ and Theorem \ref{corols}, we have $|A_{1}|\leq (1-\delta)|B_{1}|.$ Consequently, 
\[|A_{k}|\leq (1-\delta)|B_{1}|,\]
as $A_{k}\subset A_{1}.$ Now we wish to apply Lemma \ref{covering1}. Let $B\subset B_{1}$ be some ball satisfying $|B\cap A_{k+1}|>(1-\delta)|B|,$ then Corollary \ref{scaler} with $L=M^{k}$ implies that $B\subset A_{k}.$ Therefore, by applying Lemma \ref{covering1} we get 
\begin{equation}
|A_{k+1}|\leq (1-c\delta)|A_{k}|,
\end{equation}
and thus by induction, $|A_{K}|\leq (1-c\delta)^{k-1}(1-\delta)|B_{1}|=\tilde{C}M^{-\epsilon k},$ where $-\epsilon=\log(1-c\delta)/\log M$ and $\tilde{C}=(1-c\delta)^{-1}(1-\delta)|B_{1}|.$
\end{proof}
As we discussed in the introduction that we obtain the H\"{o}lder estimate as a consequence of $L^{\epsilon}$ estimate. Rather than applying it directly we use the following rescaled version of above Lemma.
\begin{corollary}\label{slepsilon}
There exist positive constants $\tilde{k}_{0}>0,~\epsilon_{1}$ and such that if $k\leq \tilde{k}_{0},$ then for any $r\leq 1,$ $\alpha\in(0,1)$ and any lower semicontinuous function $u:~\overline{B_{2r}}\longrightarrow\mathbb{R}$ satisfying 
\begin{equation*}
\left\{
\begin{aligned}
u&\geq0~~~\text{in}~~B_{2r},\\
\mathcal{L}_{\kappa}^{-}\big(D^{2}u,Du\big)&\leq \epsilon_{1}~~~\text{in}~~B_{2r},\\ 
|\{u>r^{\alpha}\}\cap B_{r}|&\geq \frac{1}{2}|B_{r}|,\\
\end{aligned}
\right.
\end{equation*}
we have $u\geq \epsilon_{1}r^{\alpha}$ in $B_{r}.$
\end{corollary}
\begin{proof}
Let $\tau$ be a universal constant such that $\tilde{C}\tau^{-\epsilon}<|B_{1}|/2,$ where $\tilde{C}$ and $\epsilon>0$ are constant from Theorem \ref{lepsilon}. Now consider the function $v(x)=\tau r^{-\alpha}u(rx).$ This function satisfies
\begin{equation*}
\left\{
\begin{aligned}
v&\geq0~~~\text{in}~~B_{2},\\
\mathcal{L}_{\kappa\tau r^{2-\alpha}}^{-}\big(D^{2}v,Dv\big)&\leq\epsilon_{1}\tau r^{2-\alpha}~~~\text{in}~~B_{2},\\ 
|\{v>\tau\}\cap B_{1}|&\geq \frac{1}{2}|B_{1}|.\\
\end{aligned}
\right.
\end{equation*}
Now, we choose $\epsilon_{1}=\tau^{-1}.$ Since $r\leq 1$ again by Remark \ref{remsc}, we have 
\[\mathcal{L}_{\kappa}^{-}\big(D^{2}v,Dv\big)\leq 1~\text{in}~B_{2}.\]
Thus by Theorem \ref{lepsilon} we find  $v>1$ in $B_{1}$ as long as $\tau r^{2-\alpha}k\leq \epsilon_{0}.$ Which is the case provided we choose $\tilde{\kappa}_{0}=\kappa_{0}\tau^{-1}=\kappa_{0}\epsilon_{1},$ where $\kappa_{0}$ is from Lemma \ref{lepaa}. The required result follows once we scale back.
\end{proof}
\subsection{Proof of Theorem \ref{mainthm}}
\begin{proof}
Following \cite{imbert2016estimates}, with $\epsilon_{1}$ as in the above corollary for any $\theta\leq 1,$ set 
\[w(x)=\frac{u(\theta x)}{C_{0}(1+\epsilon_{1}^{-1})}.\]
The above function $w$ satisfies
\begin{equation*}
\left\{
\begin{aligned}
\mathcal{L}_{\tilde{\kappa}}^{-}\big(D^{2}w,Dw\big)&\leq \epsilon_{1}~B_{1},\\ 
\mathcal{L}_{\tilde{\kappa}}^{+}\big(D^{2}w,Dw\big)&\geq-\epsilon_{1}~B_{1},\\
\|w\|_{L^{\infty}(B_{1})}&\leq 1,
\end{aligned}
\right.
\end{equation*}
where $\tilde{k}=\frac{\theta^{2}k}{C_{0}(1+\epsilon^{-1}_{1})}.$ Now we pick $\theta\leq 1$ such that 
\[\frac{\theta^{2}k}{C_{0}(1+\epsilon^{-1}_{1})}\leq \tilde{\kappa}_{0}.\]
Let us set $m_{k}=\min_{B_{2^{-k}}}w$ and $M_{k}=\max_{B_{2^{-k}}}w.$ The H\"{o}lder estimate at $0$ follows once we have 
\[M_{k}-m_{k}\leq 2\times 2^{-k\alpha},\]
for all $k.$ We follow the standard approach by induction. For $k=0$ the result is a consequence of $\|w\|_{L^{\infty}(B_{1})}\leq 1.$ Let us assume the result is true for some $k,$ that is, $M_{k}-m_{k}\leq 2\times 2^{-k\alpha}.$ Now assume that $M_{k}-m_{k}\geq 2\times 2^{-\alpha(k+1)}$ otherwise the result also follows for $(k+1).$ Observe also from definition of $m_{k+1}$ and $M_{k+1}$ it follows that $m_{k}\leq m_{k+1}$ and $M_{k+1}\leq M_{k}.$ We will get more precise monotonicity in view of the above assumption. For this let us set $a_{k}=\frac{b_{k}-a_{k}}{2}$ and consider the following two possibilities
\begin{equation*}
\begin{aligned}
&(i)~|\{w>a_{k}\}\cap B_{2^{-k-1}}|\geq\frac{|B_{2^{-k-1}}|}{2}\\
&\text{or}\\
&~(ii)~|\{w\leq a_{k}\}\cap B_{2^{-k-1}}|\geq\frac{|B_{2^{-k-1}}|}{2}.
\end{aligned}
\end{equation*}
 If $(i)$ happens we apply Corollary \ref{slepsilon} to $w-m_{k}$ with $r=2^{-k-1}$ to get $w-m_{k}\geq \epsilon_{1}2^{-(k+1)\alpha}$ for some $\epsilon_{1}.$ So we have $a_{k+1}\geq a_{k}+\epsilon_{1} 2^{-(k+1)\alpha}$ and consequently,
 \[M_{k+1}-m_{k+1}\leq M_{k}-m_{k}-\epsilon_{1}2^{-(k+1)\alpha}\leq (2^{\alpha+1}-\epsilon_{1})2^{-(k+1)\alpha}\leq 2\times 2^{-(k+1)\alpha},\]
 for any $\alpha$ satisfying $2^{1+\alpha}\leq 2+\epsilon_{1}.$\\ In the case $(ii)$ we may apply Corollary \ref{slepsilon} to $M_{k}-w$ with $r=2^{-k-1}.$ The H\"{o}lder estimate at all other points follow by scaling and translation.
\end{proof}

\section{Acknowledgement}
Second author is supported by National Board of Higher Mathematics grant no. 02011/36/2025/NBHM/RP/9466 and Anusandhan National Research Foundation-ANRF/ARGM/2025/002357/MTR.
\bibliography{ref.bib}
\bibliographystyle{abbrv}
\end{document}